\documentclass[10pt]{article}

\usepackage[hmargin=35mm,vmargin=40mm]{geometry}

\usepackage{amsmath}              
\usepackage{amssymb}              
\usepackage{amsthm}               
\usepackage{enumerate}            
\usepackage{graphicx}             
\usepackage{mathrsfs}             
\usepackage[tight]{subfigure}     
\usepackage{url}                  

\newtheorem{theorem}{Theorem}
\newtheorem{lemma}[theorem]{Lemma}
\newtheorem{corollary}[theorem]{Corollary}

\newtheorem{observation}[theorem]{Observation}
\newtheorem*{problem}{Problem}

\theoremstyle{definition}
\newtheorem*{defn}{Definition}
\newtheorem*{construction}{Construction}
\newtheorem*{remark}{Remark}

\newcommand{\coNP}{\textbf{co-NP}}
\newcommand{\fpg}[1]{\Gamma({#1})}
\newcommand{\ms}{\mathcal{M}}
\newcommand{\NP}{\textbf{NP}}
\newcommand{\problemname}[1]{\textsc{#1}}
\newcommand{\R}{\mathbb{R}}
\newcommand{\regina}{\emph{Regina}}
\newcommand{\torus}{\Theta}
\newcommand{\tri}{\mathcal{T}}

\title{The complexity of detecting taut angle structures \\ on triangulations}
\author{Benjamin A.~Burton and Jonathan Spreer}

\date{October 3, 2012} 

\begin{document}

\maketitle

\begin{abstract}
    There are many fundamental algorithmic problems on triangulated
    3-manifolds whose complexities are unknown. Here we study the
    problem of finding a taut angle structure on a 3-manifold triangulation,
    whose existence has implications for both the geometry and
    combinatorics of the triangulation. We prove that detecting taut
    angle structures is \NP-complete, but also fixed-parameter tractable in the
    treewidth of the face pairing graph of the triangulation. These
    results have deeper implications: the core techniques can serve as a
    launching point for approaching decision problems such as unknot
    recognition and prime decomposition of 3-manifolds.

    \medskip

    \noindent \textbf{Keywords}\quad
    Computational topology, triangulations,
    3-manifolds, taut structures, angle structures
\end{abstract}

%
%

\section{Introduction}
\label{sec:intro}

Much work in 3-dimensional topology is driven by algorithmic problems.
Examples include
\emph{unknot recognition} (testing whether a knot in $\R^3$ is trivial),
\emph{3-sphere recognition} (testing whether a triangulated 3-manifold
is a topological sphere),
\emph{connected sum decomposition} (decomposing a 3-manifold into
``prime'' pieces),
\emph{JSJ decomposition} (decomposing a 3-manifold into pieces with geometric
structures),
and the \emph{homeomorphism problem} (testing whether two triangulated
3-manifolds are topologically equivalent).

Many of these algorithms are new; for instance, 3-sphere recognition was
only solved in 1992 by Rubinstein \cite{rubinstein95-3sphere},
and the homeomorphism problem was only solved in 2003 with Perelman's
proof of the geometrisation conjecture \cite{kleiner08-perelman},
which ties together many complex sub-algorithms by many different authors
\cite{jaco05-lectures-homeomorphism}.  Some algorithms, such as unknot
recognition, 3-sphere recognition and connected sum decomposition, have been
implemented \cite{burton04-regina} but require exponential time;
others are currently so slow and so complex that they
have never been implemented at all.

In this paper we consider the computational complexity of problems
such as these in 3-di\-men\-sion\-al topology,
where many important questions remain wide open.
For instance, it is a major open question as to whether unknot
recognition and 3-sphere recognition can be solved in polynomial time.
Both problems are known to lie in \NP\ \cite{hass99-knotnp,schleimer11-np},
and in recent announcements both problems also lie in \coNP\ if the
generalised Riemann hypothesis holds \cite{hass12-conp,kuperberg11-conp}.
Nevertheless, current state-of-the-art algorithms for both problems
still require exponential time.

There is one prominent hardness result in this area, due to Agol, Hass
and Thurston, involving \emph{knot genus}:
if we generalise unknot recognition to
computing the genus of a knot,
\emph{and} we generalise the ambient space from
$\R^3$ to an arbitrary 3-manifold, then the problem becomes \NP-complete
\cite{agol06-knotgenus}.
The underlying proof technique also applies to problems
relating to least-area surfaces
\cite{agol06-knotgenus,dunfield11-spanning}.

Beyond the results cited above, very little is known about the
computational complexity of difficult algorithmic problems such
as these in 3-dimensional topology.

In this paper we address the problem of finding a \emph{taut angle structure}
on a triangulated 3-manifold (as outlined below).  In particular, we show that
this problem is both \NP-complete and fixed-parameter tractable.
To the authors' best knowledge,
this is the first parameterised complexity result in areas relating to
difficult 3-manifold
recognition/decomposition problems, and the first such NP-completeness result
that is not based on the Agol-Hass-Thurston construction.
More importantly, the techniques that we describe here offer a potential
launching point for obtaining such results in the related
setting of \emph{normal surface theory}, a key ingredient in all
of the decomposition and recognition problems outlined above.
We discuss these possibilities further in Section~\ref{sec:conc}.

Taut angle structures were introduced by Lackenby \cite{lackenby00-taut},
and offer a bridge between the combinatorial structure of a
triangulation and the geometric structure of the underlying manifold.
Taut angle structures
are combinatorial objects that act as limiting cases of the more
general \emph{angle structures}, as introduced by Rivin
\cite{rivin94-structures,rivin03-combopt} and Casson;
these in turn act as linear
analogues of \emph{complete hyperbolic structures}, which play an
important role in recognising and distinguishing triangulated
hyperbolic 3-manifolds.
Despite their simple discrete combinatorial description,
taut angle structures can in the right setting
lead to strict angle structures \cite{kang05-taut2}
and then complete hyperbolic structures \cite{futer11-angled},
which in general are highly desirable but also potentially elusive.

More specifically,
a \emph{taut angle structure} on a 3-manifold triangulation $\tri$
assigns interior angles
$\{0,0,0,0,\pi,\pi\}$ to the six edges of each tetrahedron of $\tri$,
so that the two $\pi$ angles are
opposite in each tetrahedron, and so that around each edge of the
overall triangulation the sum of angles is $2\pi$.
The decision problem that we study in this paper is as follows:

\begin{problem}[\problemname{taut angle structure}]
    Given an orientable 3-manifold triangulation $\tri$ with no boundary
    faces, determine whether there exists a taut angle structure on $\tri$.
    We measure the size of the input by the number of tetrahedra in
    $\tri$, which we denote by $n$.
\end{problem}

Our first main theorem is the following:

\begin{theorem}
	\label{thm:np}
    \problemname{taut angle structure} is \NP-complete.
\end{theorem}

We prove this in Section~\ref{sec:np} using a reduction from
the \NP-complete problem \problemname{monotone 1-in-3 sat}
\cite{schaefer78-sat}.  In \problemname{monotone 1-in-3 sat}
we have boolean variables $x_1,\ldots,x_t$ and clauses
of the form $x_i \vee x_j \vee x_k$, and we must determine whether the
variables can be assigned true/false
values so that one and only one of the three variables in each clause is true.

The proof involves an explicit piecewise construction of a 3-manifold
triangulation that represents a given instance of
\problemname{monotone 1-in-3 sat}.
We use three types of building blocks, which represent (i) variables $x_i$;
(ii) the duplication of variables; and (iii) clauses $x_i \vee x_j \vee x_k$.
Finding such building blocks---particularly (ii) and (iii)---was a
major challenge in constructing the proof, and was performed with
significant assistance
from the software package \regina\ \cite{burton04-regina,regina}.

In Section~\ref{sec:fpt} we present additional results on
\emph{parameterised complexity}.
Introduced by Downey and Fellows \cite{downey99-param},
parameterised complexity studies which aspects of an
\NP-complete problem make it difficult, and identifies classes of inputs
for which fast algorithms can nonetheless be found.

Our parameters are based on the \emph{face pairing graph} of the
input triangulation $\tri$ (that is,
the dual 1-skeleton of $\tri$).  Denoted $\fpg{\tri}$, the face pairing
graph is the multigraph whose nodes represent tetrahedra of $\tri$,
and whose
arcs represent pairs of tetrahedron faces that are joined together.

For \problemname{taut angle structure}, we identify two parameters of
interest: the cutwidth of $\fpg{\tri}$, and the treewidth of $\fpg{\tri}$.
We define these concepts precisely in Section~\ref{sec:prelim}, but in essence
the cutwidth measures the worst ``bottleneck'' of parallel arcs in an
optimal left-to-right layout of nodes,
and the treewidth measures how ``tree-like'' the graph is.
Our results are the following:

\begin{theorem}
	\label{thm:cutwidth}
    Let $\tri$ be a 3-manifold triangulation with $n$ tetrahedra,
    where the graph $\fpg{\tri}$ has cutwidth $\leq k$, and for which a
    corresponding layout of nodes is known.
    Then \problemname{taut angle structure} can be solved for $\tri$ in
	$O(nk \cdot 3^{3k/2})$ time.
\end{theorem}

\begin{theorem}
	\label{thm:treewidth}
    Let $\tri$ be a 3-manifold triangulation with $n$ tetrahedra,
    where the graph $\fpg{\tri}$ has treewidth $\leq k$, and for which a
    corresponding tree decomposition with $O(n)$ tree nodes is known.
    Then \problemname{taut angle structure} can be solved for $\tri$ in
    $O(nk \cdot 3^{7k})$ time.
\end{theorem}

Because $\mathrm{treewidth} \leq \mathrm{cutwidth}$
(as shown in \cite{bodlaender86-classes}), the latter result is more powerful.
Moreover, if we fix an upper bound on the treewidth $k$, there is a
known linear-time algorithm to test whether a graph has
treewidth $\leq k$ and, if so, to compute
a corresponding tree decomposition with $O(n)$ tree nodes
\cite{bodlaender96-linear}.
Therefore Theorem~\ref{thm:treewidth} shows that,
in the case of bounded treewidth,
we can solve \problemname{taut angle structure}
in \emph{linear time} in the input size $n$.  That is:

\begin{corollary}
    \label{cor:fpt}
    \problemname{taut angle structure} is linear-time fixed-parameter tractable,
    where the parameter is taken to be the treewidth of the face pairing
    graph of the input triangulation.
\end{corollary}

For 3-manifold triangulations
the treewidth of $\fpg{\tri}$ is a natural parameter, and there
are well-known families of triangulations for which the treewidth
remains small.
Moreover, our fixed-parameter tractability result is consistent with
experimental observations from running other, more complex algorithms
over small-treewidth triangulations.
We discuss these issues further in Section~\ref{sec:conc}.

Throughout this paper we work in the word RAM model, where
simple arithmetical operations on $(\log n)$-bit integers are assumed
to take constant time.

\section{Preliminaries}
\label{sec:prelim}

\subsection{Triangulations}

By a \emph{3-manifold triangulation}, we mean a collection of $n$ abstract
tetrahedra, some or all of whose faces are affinely identified or ``glued
together'' in pairs.  As a consequence of these face gluings, many
tetrahedron edges may become identified together; we refer to the result as
a single \emph{edge of the triangulation}, and likewise with vertices.

This is a purely combinatorial definition: there are no geometric constraints
(such as embeddability in some $\R^d$), and the result need not be a
simplicial complex.  We may glue together two faces of the same
tetrahedron if we like.  A single edge of the triangulation might appear
as multiple edges of the same tetrahedron, and likewise with vertices.
It is common to work with \emph{one-vertex triangulations}, where all
vertices of all tetrahedra become identified as a single point.

The only constraints are the following.  Each tetrahedron face must be
identified with one and only one partner (we call these
\emph{internal faces}), or with nothing at all (we call these
\emph{boundary faces}).
Moreover, no edge may be identified with itself in reverse as a result
of the face gluings.
Any edge on a boundary face is called a \emph{boundary edge},
and all others are called \emph{internal edges}.

The \emph{link} of a vertex $V$ of the triangulation is the frontier of
a small regular neighbourhood of $V$.  If the link of $V$ is a closed
surface but not a sphere, we call $V$ an \emph{ideal vertex}.
Any triangulation with one or more ideal vertices
is called an \emph{ideal triangulation}.

Although the neighbourhood of an ideal vertex is not locally $\R^3$
(and so ideal triangulations do not represent 3-manifolds per se),
topologists often use ideal triangulations as an economical way
to represent 3-manifolds with boundary (obtained by truncating the
ideal vertices)
or non-compact 3-manifolds (obtained by deleting the ideal vertices).
Because of this, ideal triangulations are ubiquitous in the study of
hyperbolic 3-manifolds.

\begin{figure}[htb]
    \centering
    \subfigure[]{\label{fig:fig8}%
        \includegraphics[width=.41\textwidth]{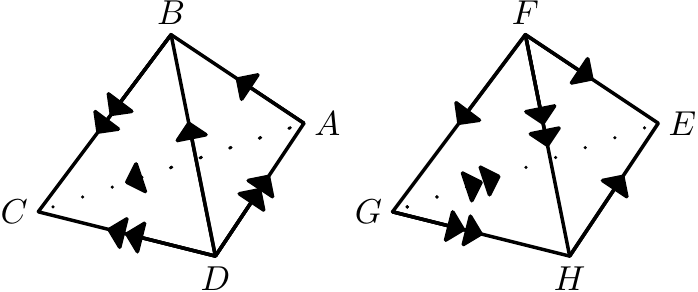}}
    \hspace{1.5cm}
    \subfigure[]{\label{fig:fpg}%
        \includegraphics[width=.21\textwidth]{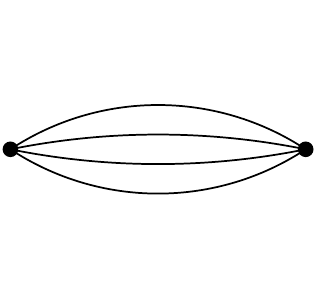}}
    \caption{The figure eight knot complement and its face pairing graph}
\end{figure}

Figure~\ref{fig:fig8} illustrates Thurston's famous ideal triangulation
of the figure eight knot complement \cite{thurston78-lectures}.
There are $n=2$ tetrahedra, labelled
$\mathit{ABCD}$ and $\mathit{EFGH}$, with the following face gluings:
\[
    \mathit{ABC} \longleftrightarrow \mathit{FGE}; \qquad
    \mathit{ABD} \longleftrightarrow \mathit{HEF}; \qquad
    \mathit{ACD} \longleftrightarrow \mathit{HEG}; \qquad
    \mathit{BCD} \longleftrightarrow \mathit{GHF}.
\]
As a consequence of these face gluings, we obtain two edges of the
triangulation, indicated by the two types of arrowhead in the diagram.
All vertices of all tetrahedra become identified as a single
ideal vertex of the triangulation, whose link is a torus.
It can be shown that truncating this vertex does indeed yield the
figure eight knot complement (i.e., the 3-manifold with torus boundary
obtained by deleting a small neighbourhood of the figure eight knot
from the 3-sphere).%
    \footnote{To highlight the efficiency of ideal triangulations:
    the smallest known \emph{non-ideal} triangulation of the figure eight knot
    complement (using boundary faces instead of an ideal vertex)
    requires $n=10$ tetrahedra.}

The \emph{size} of a triangulation is measured by
the number of tetrahedra $n$.  To input a triangulation, one
presents the list of face gluings (as illustrated above), which
requires $O(n \log n)$ bits.

\subsection{Taut angle structures}

Let $\tri$ be a 3-manifold triangulation with no boundary faces.
A \emph{taut angle structure} on $\tri$ assigns interior angles
$\{0,0,0,0,\pi,\pi\}$ to the six edges of each tetrahedron of $\tri$,
so that the two $\pi$ angles are
opposite in each tetrahedron, and so that around each edge of the
triangulation the sum of angles is $2\pi$.
Geometrically, a taut structure shows how the tetrahedra can be consistently
``flattened'' throughout the triangulation.
Here we use the nomenclature of
Hodgson et~al.\ \cite{hodgson11-veering}---our taut angle structures
are slightly more general than
the original taut structures of Lackenby \cite{lackenby00-taut}, who also
requires consistent coorientations on the 2-faces of the triangulation.

To illustrate, we can place a taut angle structure on
Figure~\ref{fig:fig8} by assigning $\pi$ to the opposite edges $\mathit{AC}$
and $\mathit{BD}$ of the first tetrahedron, and to the opposite edges
$\mathit{EG}$ and $\mathit{FH}$ of the second tetrahedron.  It is easily
seen that both edges of the triangulation (the single arrowheads versus the
double arrowheads) receive the angle $\pi$ exactly twice each.

We refer to the two $\pi$ edges in each tetrahedron as \emph{marked}.
Combinatorially, a taut structure simply involves choosing two opposite
edges of each tetrahedron to mark, in such a way that every edge of the
triangulation is marked exactly twice.

A simple Euler characteristic calculation shows that, in a triangulation
with no boundary faces, a taut angle structure can only exist if every
vertex link is a torus or a Klein bottle.  That is, taut angle
structures require \emph{ideal triangulations}.

Here we generalise this definition to support triangulations with boundary
(which become important as we piece together triangulations for our
\NP-completeness proof).
If $\tri$ is any 3-manifold triangulation (with or without boundary
faces), then a taut angle structure on $\tri$ involves choosing two opposite
edges of each tetrahedron to mark,
so that every internal edge of the triangulation is marked
\emph{exactly} twice, and every boundary edge of the
triangulation is marked \emph{at most} twice.

Let $\tri$ and $\tri'$ be 3-manifold triangulations for which
$\tri$ is a subcomplex of $\tri'$
(i.e., $\tri'$ is obtained from $\tri$
by adding new tetrahedra and/or additional face gluings).
If $\tau$ and $\tau'$ are taut angle structures on $\tri$ and $\tri'$
respectively, we say that $\tau'$ \emph{extends} $\tau$ if they both
assign the same interior angles to the tetrahedra from $\tri$
(i.e., the tetrahedra that belong to both triangulations).

\subsection{Face pairing graphs}

The \emph{face pairing graph} of a 3-manifold triangulation $\tri$,
denoted $\fpg{\tri}$,
is the multigraph whose nodes represent tetrahedra, and whose arcs
represent pairs of tetrahedron faces that are glued together.
A face pairing graph may contain loops (if two faces of the same
tetrahedron are glued together), and/or multiple edges (if two tetrahedra
are joined together along more than one face).

If every face of $\tri$ is internal, then $\fpg{\tri}$ is a 4-valent
graph.  Figure~\ref{fig:fpg} shows the face pairing graph
of the figure eight knot complement as presented in Figure~\ref{fig:fig8}.

In our parameterised complexity analysis, we measure both the cutwidth and
the treewidth of $\fpg{\tri}$.  These concepts are defined as follows
\cite{downey99-param,korach93-width}:

\begin{defn}[Cutwidth]
    A {\em cut} of a graph $G$ is a partition of its nodes
    into two disjoint subsets $N_1$ and $N_2$. The set of arcs
    with one endpoint in $N_1$ and the other in $N_2$ is called the
    {\em cutset}, and the number of arcs in the cutset is referred to as
    the {\em width} of the cut $(N_1,N_2)$.

    The {\em cutwidth} of $G$ is the smallest $k$ for which
    there exists an ordering (or \emph{layout})
    $\nu_1 , \ldots , \nu_n$ of the nodes of $G$
    such that the width of every cut
    $(\{\nu_{1} , \ldots \nu_{i} \},\{\nu_{i+1} , \ldots \nu_{n} \})$
    is at most $k$.
\end{defn}

\begin{defn}[Treewidth]
    A {\em tree decomposition} of a graph $G$ is a tree $T$ and a
    collection of \emph{bags} $\{X_i\,|\,\mbox{$i$ is a node of $T$}\}$.
    Each bag $X_i$ is a subset of nodes of $G$, and we require:
    (i)~every node of $G$ is contained in at least one bag $X_i$;
    (ii)~for each edge of $G$, some bag $X_i$ contains both its endpoints;
    and
    (iii)~for all nodes $i,j,k$ of $T$, if $j$ lies on the unique path from
    $i$ to $k$ in $T$, then $X_i \cap X_k \subseteq X_j$.

    The \emph{width} of a tree decomposition is defined as $\max |X_i| - 1$,
    and the {\em treewidth} of $G$ is the minimum width over all tree
    decompositions.
\end{defn}

In essence, cutwidth measures the worst ``bottleneck'' of parallel
arcs in an optimal left-to-right layout of nodes that is chosen to make
this bottleneck as small as possible, and
treewidth measures how far $G$ is from being a tree
(in particular, a tree always has treewidth 1).
Bodlaender shows that
$\mathrm{cutwidth} \geq \mathrm{treewidth}$
\cite{bodlaender86-classes}; 
on the other hand, there are graphs with bounded treewidth and arbitrarily
large cutwidth, and so these two parameters measure genuinely different
features.

Computing cutwidth and treewidth are both
NP-complete \cite{arnborg87-embeddings,garey79-intractability}.
However, for fixed $k$ it can be decided in linear time whether a given
graph has cutwidth $\leq k$ and/or
treewidth $\leq k$ \cite{bodlaender96-linear,thilikos05-cutwidth}.

\section{NP-completeness}
\label{sec:np}

In this section we prove Theorem~\ref{thm:np}, i.e., that
\problemname{taut angle structure} is \NP-complete.
As stated earlier, we do this using a reduction from the
\NP-complete problem \problemname{monotone 1-in-3-sat}
\cite{schaefer78-sat}.

The overall structure of the proof is as follows.
Throughout this section, let $\mathcal{M}$ be a given instance of
\problemname{monotone 1-in-3 sat}, with $t$ variables
$x_1,\ldots,x_t$, and $c$ clauses each of the form $x_i \vee x_j \vee x_k$.
We say that $\ms$ is \emph{solvable} if and only if there is some
assignment of true/false values to the variables so that exactly one of the
three terms in each clause is true.

Our strategy is to build a corresponding triangulation $\tri_\ms$
that has a taut angle structure if and only if $\mathcal{M}$ is solvable.
We build $\tri_\ms$ by hooking together three types of gadgets,
all of which are triangulations with boundary faces:
(i)~\emph{variable gadgets}, each with two choices
of taut angle structure that represent true or false respectively
for a single variable $x_i$ of $\mathcal{M}$;
(ii)~\emph{fork gadgets} that allow us to propagate this choice
for $x_i$ to several clauses simultaneously; and
(iii)~\emph{clause gadgets} that connect three variable gadgets
and support
an overall taut angle structure if and only if precisely one of the
three corresponding variable choices is true.

These gadgets have 2, 21 and 4 tetrahedra respectively, and we describe
and analyse them in Sections~\ref{sec:variable}, \ref{sec:fork} and
\ref{sec:clause}.  We then finish off the proof of
Theorem~\ref{thm:np} in Section~\ref{sec:npproof}, which is
a simple matter of hooking the gadgets together and observing that
the entire construction can be done in polynomial time.

We hook the gadgets together along tori: each such torus consists of two
faces, three edges and one vertex.  To facilitate lemmas and proofs,
we assign \emph{types} $a$, $b$ and $c$ to
the three edges of each such torus $\torus$,
as illustrated in Figure~\ref{fig:twoFaceTorusBdry} (we explicitly
describe these edge types for each gadget).

\begin{figure}[htb]
    \centering
    \includegraphics[width=0.15\textwidth]{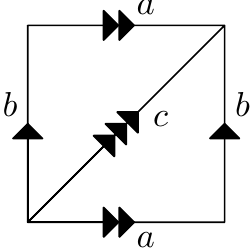}
    \caption{A two-face torus $\torus$ with edge types $a$, $b$ and $c$}
    \label{fig:twoFaceTorusBdry}
\end{figure}

For any taut angle structure $\tau$, the \emph{boundary pattern} of $\tau$ on
the torus $\torus$ is the triple $(m_a,m_b,m_c)$, where $m_a$, $m_b$ and
$m_c$ count the number of markings on
the edges of type $a$, $b$ and $c$
respectively.  By definition of a taut angle structure,
each of $m_a,m_b,m_c \in \{0,1,2\}$.
We use boundary patterns to represent true/false
values of variables in $\ms$: in particular, the boundary pattern
$(2,0,0)$ represents \emph{true}, and the boundary pattern
$(0,2,0)$ represents \emph{false}.

\subsection{The variable gadget}
\label{sec:variable}

A variable gadget is a triangulation with torus boundary that
has precisely two taut angle structures: one with boundary pattern
$(2,0,0)$ (representing \emph{true}),
and the other with boundary pattern $(0,2,0)$
(representing \emph{false}).
We first define the gadget, and then prove the necessary properties.

The construction is simple: we use a $(1,3,4)$ \emph{layered solid torus},
a two-tetrahedron instance of a more general and much-studied family
of solid torus triangulations \cite{jaco03-0-efficiency}.
The details are as follows.

\begin{construction}[Variable gadget]
    To build a variable gadget, we begin with the tetrahedron
    $\Delta_1$ whose vertices are labelled $A,B,C,D$, and we identify
    faces $\mathit{ABD}$ and $\mathit{BDC}$ (the rear faces in the diagram).
    This is the well-known one-tetrahedron triangulation of the solid
    torus \cite{burton04-facegraphs,jaco03-0-efficiency},
    and has three boundary edges
    $\mathit{AB}=\mathit{BD}=\mathit{DC}$; $\mathit{AD}=\mathit{BC}$;
    and $\mathit{CA}$, as illustrated in Figure~\ref{fig:var1}.

    \begin{figure}[htb]
        \centering
        \subfigure[]{\label{fig:var1}%
            \includegraphics[width=.21\textwidth]{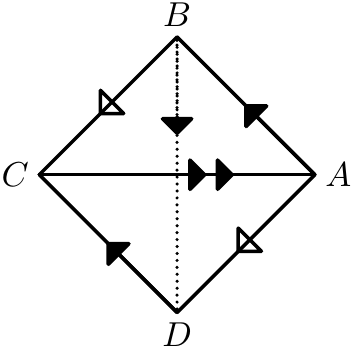}}
        \hspace{1.5cm}
        \subfigure[]{\label{fig:varlayer}%
            \includegraphics[width=.54\textwidth]{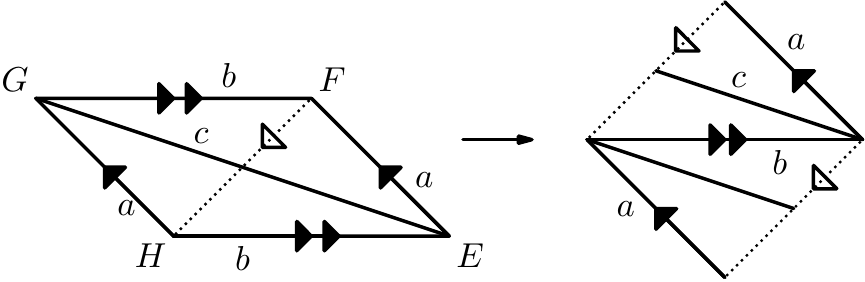}}
        \caption{Building a variable gadget}
    \end{figure}

    We now take a second tetrahedron $\Delta_2$ with vertices labelled
    $E,F,G,H$, and glue the remaining boundary faces of $\Delta_1$ to $\Delta_2$
    by identifying $\mathit{ABC}$ with $\mathit{EFH}$ and $\mathit{ACD}$ with
    $\mathit{FGH}$ as indicated in Figure~\ref{fig:varlayer}.
    This is the two-tetrahedron $(1,3,4)$ layered solid torus,
    with three boundary edges and one internal edge.
    On the boundary we assign edge types
    $a \to \mathit{AB} = \mathit{BD} = \mathit{DC} = \mathit{EF} = \mathit{HG}$;
    $b \to \mathit{CA} = \mathit{HE} = \mathit{GF}$; and
    $c \to \mathit{EG}$.
    For completeness, the one internal edge is
    $\mathit{AD} = \mathit{BC} = \mathit{FH}$.
\end{construction}

\begin{observation} \label{obs:lst}
    The variable gadget is a layered solid torus,
    as described in \cite{jaco03-0-efficiency}.
    In particular, it has just one vertex, whose link is a disc,
    and its two boundary faces join together to form a torus.
\end{observation}

These properties are true of any layered solid torus; see
\cite{burton04-facegraphs,jaco03-0-efficiency} for
more information on the general layered solid torus
construction (the details of which are not important here).
All of the claims above can also be verified computationally using the
software package {\regina} \cite{burton04-regina,regina}.

\begin{lemma} \label{l:variable}
    The variable gadget supports precisely two taut angle
    structures: one with boundary pattern $(2, 0, 0)$, and one with
    boundary pattern $(0, 2, 0)$.
\end{lemma}

We give a theoretical proof here;
however, again this is easy to verify computationally
using \regina, which can enumerate
all taut angle structures on a given triangulation.

\begin{proof}
    Consider the internal edge $\mathit{AD}=\mathit{BC}=\mathit{FH}$.
    Since edges $\mathit{AD}$ and $\mathit{BD}$ are opposite in
    tetrahedron $\Delta_1$, the only way to mark this internal edge
    \emph{twice} is to mark edges $\mathit{AD}$ and $\mathit{BC}$ of
    tetrahedron $\Delta_1$, and to not mark edge $\mathit{FH}$
    in tetrahedron $\Delta_2$.

    Therefore our choice for $\Delta_1$ is forced, and we are left with
    two options for what to mark in $\Delta_2$.
    We could either mark $\mathit{EF}$ and $\mathit{HG}$, which
    yields a taut angle structure with boundary pattern $(2,0,0)$,
    or we could mark $\mathit{EH}$ and $\mathit{FG}$,
    which yields a taut angle structure with boundary pattern $(0,2,0)$.
\end{proof}

\subsection{The fork gadget}
\label{sec:fork}

A fork gadget is a triangulation
that allows us to duplicate a variable $x_i$ from our
\problemname{monotone 1-in-3 sat} instance $\ms$.
Specifically, we can attach a fork gadget to some boundary torus $\torus$
of some triangulation $\tri$ with a taut angle structure $\tau$;
as a result it produces two new boundary tori that both inherit the
\emph{same} boundary pattern with which $\tau$ meets $\torus$.
The details are as follows.

\begin{construction}[Fork gadget]
    To build a fork gadget, we begin with an annular prism; that is,
    the prism over a disc with a hole cut out of the centre, as
    illustrated in Figure~\ref{fig:forkprism}.
    We triangulate this prism with 21 tetrahedra: the precise triangulation is
    important, and is spelled out explicitly in the appendix.
    As a consequence, this triangulates the outer cylinder with four triangles
    and four vertices $A,B,C,D$, and triangulates the inner cylinder with
    two triangles and two vertices $E,F$.

    We then glue the top of the prism to the bottom, effectively creating a
    ``hollow'' solid torus; that is, a manifold with a torus boundary component
    on the outside and another torus boundary component on the inside.
    We assign edge types $a,b,c$ to the three edges of the inner torus, and
    also to the six edges of the outer torus;
    the precise labellings are shown in Figure~\ref{fig:forkglue}.

    \begin{figure}[htb]
        \centering
        \subfigure[]{\label{fig:forkprism}%
            \includegraphics[width=.25\textwidth]{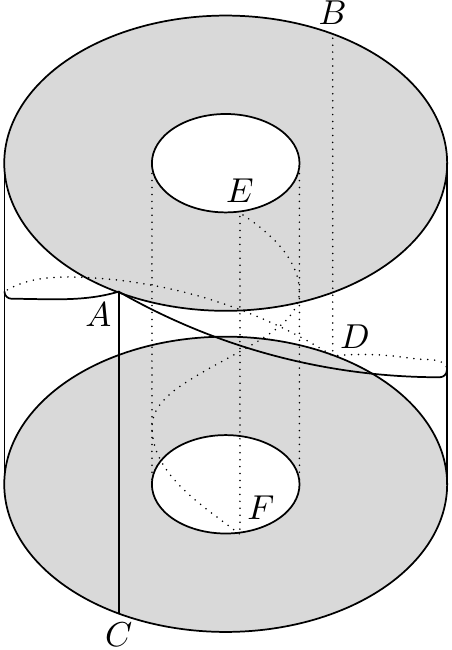}}
        \hspace{2cm}
        \subfigure[]{\label{fig:forkglue}%
            \includegraphics[width=.44\textwidth]{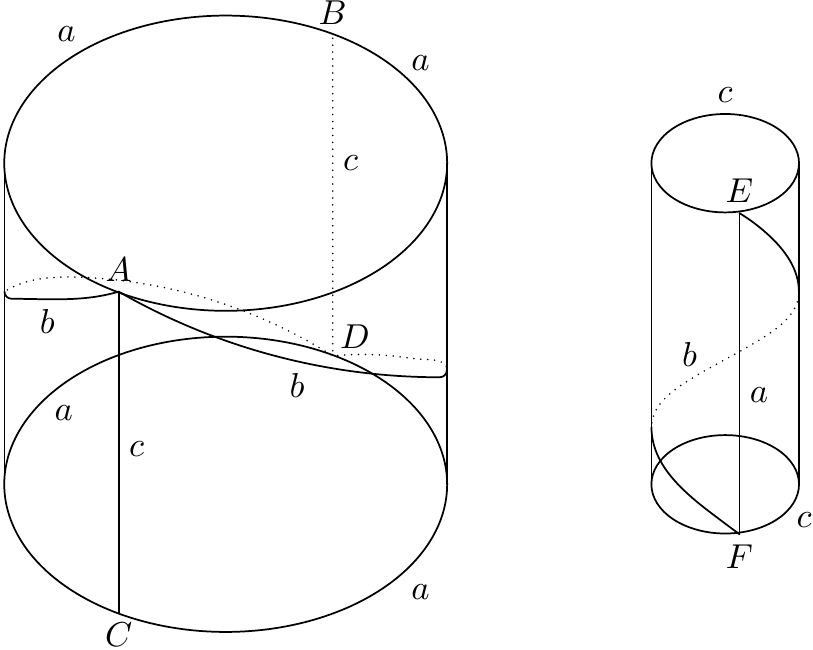}}
        \caption{Building a fork gadget}
    \end{figure}
\end{construction}

\begin{lemma} \label{l:fork}
    Let $\tri$ be a 3-manifold triangulation, two of whose
    boundary faces form a two-triangle torus $\torus$
    with the usual $a,b,c$ edge types.
    Let $\tri'$ be the new triangulation obtained by attaching a fork
    gadget to $\torus$ along the right-hand side of the
    outer torus of the fork gadget,
    as illustrated in Figure~\ref{fig:forkattach},
    so that the edge types $a,b,c$ match.
    Then $\tri'$ is a 3-manifold triangulation with four new boundary
    faces that form two disjoint two-triangle tori $\torus',\torus''$, as
    illustrated in Figure~\ref{fig:forkattachnew}.

    \begin{figure}[htb]
        \centering
        \subfigure[]{\label{fig:forkattach}%
            \includegraphics[width=.42\textwidth]{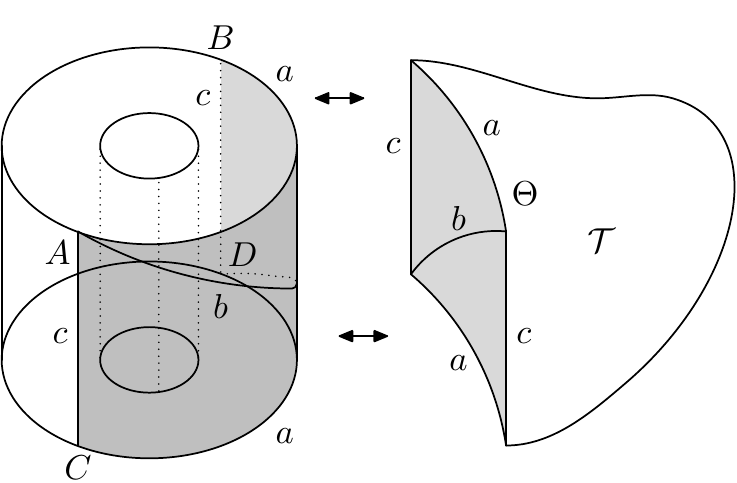}}
        \hspace{1cm}
        \subfigure[]{\label{fig:forkattachnew}%
            \includegraphics[width=.38\textwidth]{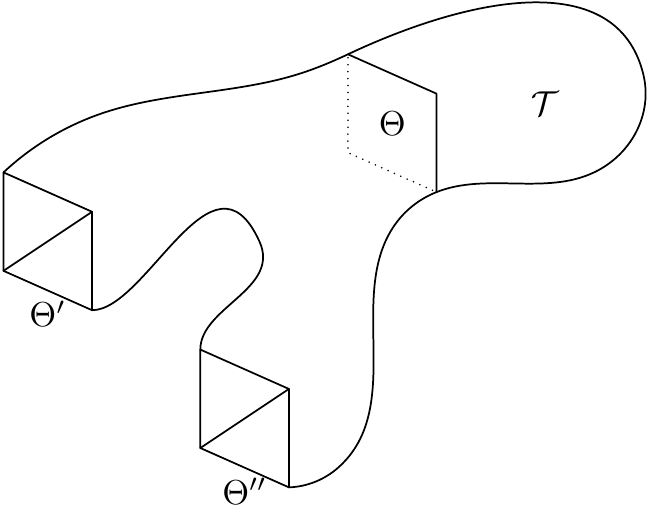}}
        \caption{Attaching a fork gadget to the 3-manifold triangulation $\tri$}
    \end{figure}

    Let $\tau$ be a taut angle structure on $\tri$ that meets the original
    boundary torus $\torus$ in one of the patterns $(2,0,0)$ or $(0,2,0)$.
    Then we can extend $\tau$ through the fork gadget to obtain a
    taut angle structure $\tau'$ on $\tri'$.  Moreover, every such extension
    $\tau'$ meets the new boundary tori $\torus',\torus''$ in the same
    boundary pattern with which $\tau$ meets $\torus$.
\end{lemma}

In other words: a fork gadget allows us to duplicate a boundary torus in a
way that also \emph{duplicates the boundary patterns} of taut angle structures.
Because the fork gadget contains 21 tetrahedra, a theoretical analysis of the
possible taut angle structures would be onerous;
therefore we use the software package \regina\ to assist with our analysis.

\begin{proof}
    It is simple to see that the new triangulation $\tri'$ satisfies our
    conditions for a 3-manifold triangulation: the only new edge
    identification that results from the gluing the fork gadget to the
    torus $\torus$ is that the vertical edges
    $\mathit{AC}$ and $\mathit{BD}$ on the outer cylinder
    become identified together (note that the upper
    right-hand edge $\mathit{AB}$ and the lower right-hand edge
    $\mathit{CD}$ are already identified, and the gluing to $\torus$
    is consistent with this).  In particular, no edge becomes glued to
    itself in reverse.

    The inner cylinder (with vertices $E=F$) remains unchanged, and
    becomes the first new torus boundary component $\torus'$ (recall
    again that the upper edge of the inner cylinder is already glued to
    the lower edge within the fork gadget).
    As the vertical edges $\mathit{AC}$ and $\mathit{BD}$ become
    identified, the two remaining faces on the left-hand side of the
    outer cylinder form a second boundary torus $\torus''$, with
    one horizontal edge $a \to \mathit{AB} = \mathit{CD}$,
    one diagonal edge $b \to \mathit{AD}$, and one vertical edge
    $c \to \mathit{AC} = \mathit{BD}$.

    We now examine how the taut structure $\tau$ can be extended through
    the fork gadget.  By entering Table~\ref{tab:forkfull} as a
    triangulation into the software package
    \regina\ \cite{burton04-regina,regina}\footnote{%
        Readers are welcome to download {\regina} and try this
        for themselves.  However, a word of caution: {\regina} numbers its
        tetrahedra and vertices starting from $0$, not $1$.  Therefore
        all tetrahedron labels and all vertex numbers
        in Table~\ref{tab:forkfull} must be reduced by $1$.}
    and enumerating all
    taut angle structures, we find that the fork gadget has precisely
    four taut angle structures $\tau_1,\tau_2,\tau_3,\tau_4$.
    Table~\ref{tab:taut} lists the number of times that each $\tau_i$
    marks each edge on the inner and outer cylinder of the fork gadget;
    the edges are identified by their labels as shown in
    Figures~\ref{fig:forkglue} and \ref{fig:forkattach}.

    \begin{table}[tb]
        \centering
        \begin{tabular}{l|c|c|c|c}
            Edge & $\tau_1$ & $\tau_2$ & $\tau_3$ & $\tau_4$ \\
            \hline
            \multicolumn{5}{l}{\emph{Inner cylinder}} \\
            \hline
            Vertical edge ($a$)   &0&0&2&2\\
            Diagonal edge ($b$)   &2&2&0&0\\
            Horizontal edge ($c$) &0&0&0&0\\
            \hline
            \multicolumn{5}{l}{\emph{Outer cylinder}} \\
            \hline
            Horizontal left edge ($a$, on $\torus''$)   &1&0&2&1\\
            Horizontal right edge ($a$, meets $\torus$) &1&2&0&1\\
            Diagonal left edge ($b$, on $\torus''$)     &1&2&0&1\\
            Diagonal right edge ($b$, meets $\torus$)   &1&0&2&1\\
            Vertical front edge ($c \to \mathit{AC}$)   &0&0&0&0\\
            Vertical rear edge ($c \to \mathit{BD}$)    &0&0&0&0
        \end{tabular}
        \caption{Edge markings from the four taut angle structures
        on the fork gadget}
        \label{tab:taut}
    \end{table}

    From here the proof is a simple matter of chasing edge markings
    around the diagram.
    \begin{itemize}
        \item Suppose that $\tau$ meets the torus
        $\torus$ in the boundary pattern $(2,0,0)$.
        
        Consider edge $a$ on the torus $\torus$, which becomes an
        internal edge of the final triangulation $\tri'$.
        Since edge $a$ on $\torus$ already has two markings, and since
        this edge is joined to the horizontal right edge $a$ on the
        outer cylinder of the fork gadget, this latter edge must have
        \emph{zero} markings within the fork gadget.  This means that
        the only compatible taut angle structure within the fork gadget
        is $\tau_3$.

        Likewise, edge $b$ on the torus $\torus$ becomes an internal
        edge of $\tri'$.  Since edge $b$ on $\torus$ has no markings,
        it requires two markings from within the fork gadget;
        we see from Table~\ref{tab:taut} that $\tau_3$ provides this
        as required.

        The vertical edge $c$ on the torus $\torus$ becomes the boundary
        edge $c$ of the new boundary torus $\torus''$.  Since this edge
        receives no markings from either $\torus$ or $\tau_3$, it has no
        markings in the final triangulation $\tri'$.

        This shows that combining $\tau_3$ with $\tau$ gives us a taut
        angle structure on $\tri'$; that is, $\tau$ can indeed be
        extended through the fork gadget (and this is the only one way of
        doing so).  We now examine the boundary patterns that arise on
        the new boundary tori $\torus'$ and $\torus''$.

        We have already seen above that edge $c$ on the torus $\torus''$
        receives no markings at all.
        The remaining edges of $\torus'$ and $\torus''$ are all new to the
        fork gadget, and so any markings on them must come from $\tau_3$.
        Reading these figures from
        Table~\ref{tab:taut}, we see that the inner torus $\torus'$
        receives a final boundary pattern of $(2,0,0)$, and the outer
        torus $\torus''$ likewise receives a final boundary pattern
        of $(2,0,0)$.

        \item Suppose instead that $\tau$ meets the torus $\torus$
        in the boundary pattern $(0,2,0)$.

        We can follow a similar argument as before.
        This time edge $b$ on $\torus$ already has two markings, and so
        the diagonal right edge $b$ on the outer cylinder of the fork
        gadget must have no markings within the fork gadget, forcing us
        to choose $\tau_2$.  As before we see that the markings from
        $\tau_2$ are consistent on the gluing torus $\torus$, and leave
        the new boundary tori $\torus'$ and $\torus''$ with boundary
        patterns $(0,2,0)$ and $(0,2,0)$ respectively.
    \end{itemize}

    Therefore any such taut angle structure $\tau$ on $\tri$ can be
    extended through the fork gadget, and the resulting boundary patterns on
    both new tori $\torus'$ and $\torus''$ will be identical to the original
    boundary pattern on $\torus$.
\end{proof}

\begin{remark}
    Constructing the fork gadget,
    and in particular finding the ``right'' triangulation of the annular
    prism,
    was the most difficult aspect of
    this paper.  It involved an interplay
    between theory and computation, and
    the triangulation includes substructures explicitly
    tailored to eliminate unwanted extensions of $\tau$.
    See the full version of this paper for details.
\end{remark}

\subsection{The clause gadget}
\label{sec:clause}

A clause gadget represents a clause $x_i \vee x_j \vee x_k$ from our
\problemname{monotone 1-in-3 sat} instance $\ms$.
We can attach it to three boundary tori of some
triangulation $\tri$ with a taut angle structure $\tau$, and $\tau$
will only extend through the clause gadget if exactly one of its
boundary patterns on the tori is $(2,0,0)$ (\emph{true}),
and the other two are $(0,2,0)$ (\emph{false}).
The details are as follows.

\begin{construction}[Clause gadget]
    To build a clause gadget, we begin with a two-triangle torus
    $\mathit{ABCD}$ and cone it to a point using two tetrahedra,
    as illustrated in Figure~\ref{fig:clausecone}
    (so the upper and lower faces $\emph{ABX}$ and $\emph{DCX}$ are
    joined, as are the left and right faces $\emph{ADX}$ and $\emph{BCX}$).
    This makes the cone point $X$ an ideal vertex (its link is a torus).
    There are two boundary faces remaining (at the front of the diagram);
    to each we
    attach a new tetrahedron, as shown in Figure~\ref{fig:clausetets}.
    We now have six boundary faces that together form a torus,
    which concludes our construction.

    \begin{figure}[htb]
        \centering
        \subfigure[]{\label{fig:clausecone}%
            \includegraphics[width=.25\textwidth]{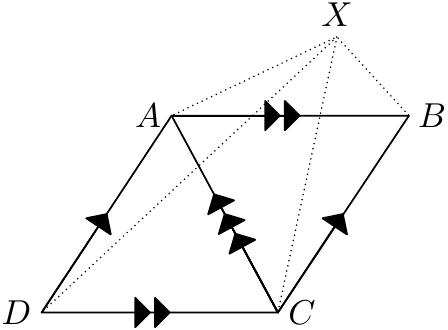}}
        \hspace{.7cm}
        \subfigure[]{\label{fig:clausetets}%
            \includegraphics[width=.21\textwidth]{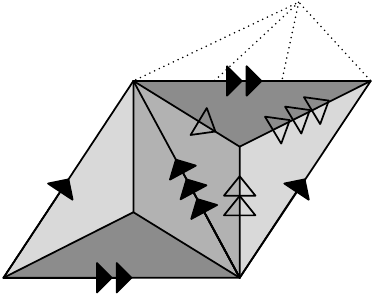}}
        \hspace{1.3cm}
        \subfigure[]{\label{fig:clausehex}%
            \includegraphics[width=.21\textwidth]{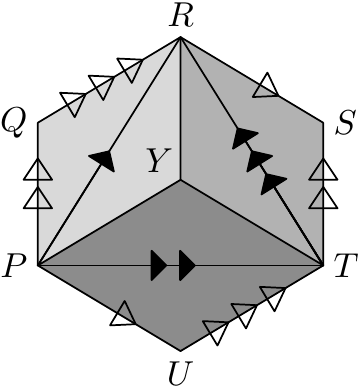}}
        \caption{Building a clause gadget}
    \end{figure}

    For convenience, Figure~\ref{fig:clausehex} presents this
    same boundary torus as a hexagon $\mathit{PQRSTU}$;
    this will make the attaching process easier to describe
    in Lemma~\ref{l:clause} below.
\end{construction}

\begin{lemma} \label{l:clause}
    Let $\tri$ be a 3-manifold triangulation (possibly disconnected),
    six of whose boundary faces form three disjoint two-triangle tori
    $\torus_1$, $\torus_2$, $\torus_3$ each with the usual $a,b,c$ edge types.
    Let $\tri'$ be the new triangulation obtained by attaching these
    three tori to a clause gadget along the rectangles
    $\mathit{PQRY}$, $\mathit{RSTY}$ and $\mathit{TUPY}$ respectively,
    as illustrated in Figure~\ref{fig:clauseattach}; in particular,
    so the type~$a$ edges join to
    $\mathit{PR}$, $\mathit{RT}$ and $\mathit{TP}$,
    and the type~$b$ edges join to
    $\mathit{QR}$ and $\mathit{PY}$ for $\torus_1$,
    $\mathit{ST}$ and $\mathit{RY}$ for $\torus_2$, and
    $\mathit{UP}$ and $\mathit{TY}$ for $\torus_3$.
    Then $\tri'$ is a 3-manifold triangulation.

    \begin{figure}[htb]
        \centering
        \includegraphics[width=.5\textwidth]{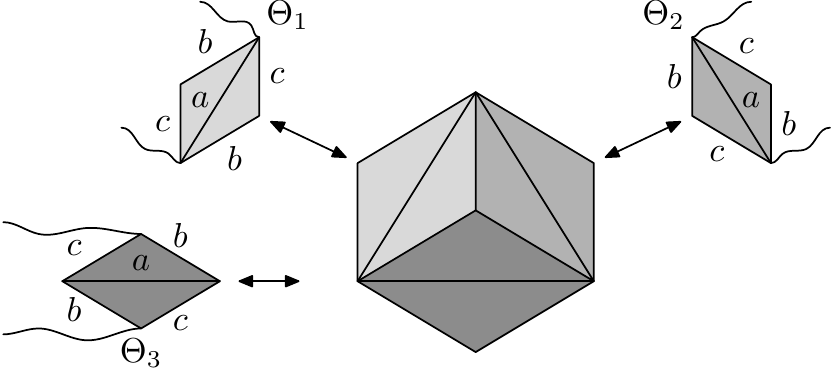}
        \caption{Attaching a clause gadget to three boundary tori of $\tri$}
        \label{fig:clauseattach}
    \end{figure}

    Let $\tau$ be a taut angle structure on $\tri$ that meets each torus
    $\torus_1,\torus_2,\torus_3$ in one of the patterns $(2,0,0)$ or $(0,2,0)$.
    Then we can extend $\tau$ through the clause gadget to obtain a taut
    angle structure $\tau'$ on $\tri'$ if and only if exactly one
    of these three boundary patterns is $(2,0,0)$.
\end{lemma}

We give a theoretical proof here, but again
this can also be verified computationally.

\begin{proof}
    Once again it is simple to see that $\tri'$ satisfies our
    conditions for a 3-manifold triangulation:
    as a result of the torus gluings, we obtain new edge
    identifications on the clause gadget as follows:
    $\mathit{QP} = \mathit{RY} = \mathit{ST}$;
    $\mathit{QR} = \mathit{PY} = \mathit{UT}$; and
    $\mathit{RS} = \mathit{YT} = \mathit{PU}$.  None of these
    identifications cause an edge to be glued to itself in reverse.

    Now let $\tau$ be a taut angle structure on $\tri$ as described in
    the statement of the lemma.  We examine how we might extend $\tau$
    through the clause gadget.

    \begin{itemize}
        \item
        Consider the three edges $\mathit{AB}$, $\mathit{BC}$ and
        $\mathit{CA}$.  These become three distinct internal edges after the
        gluings, and so they require six markings between them.
        However, every pair of opposite edges in every tetrahedron from the
        clause gadget only includes one of these three
        edges, and so between them
        these three edges can only receive at most four markings from within the
        clause gadget itself.

        Therefore edges $\mathit{AB}$, $\mathit{BC}$ and $\mathit{CA}$ must
        receive at least two markings from the external structure $\tau$.
        Since $\mathit{AB}$, $\mathit{BC}$ and $\mathit{CA}$ are all joined
        to edges of type $a$ on the tori $\torus_1,\torus_2,\torus_3$, it
        follows that at least one of the boundary patterns on these
        three tori must be $(2,0,0)$.

        \item
        Now consider the three edges
        $\mathit{QP} = \mathit{RY} = \mathit{ST}$,
        $\mathit{QR} = \mathit{PY} = \mathit{UT}$ and
        $\mathit{RS} = \mathit{YT} = \mathit{PU}$.
        Again these become three distinct internal edges after the
        gluings, and so again they require six markings between them.
        However, this time they can only receive \emph{two} markings
        from within the clause gadget: these edges do not meet the
        ``cone tetrahedra'' from Figure~\ref{fig:clausecone} at all,
        and the two tetrahedra that we attach in
        Figure~\ref{fig:clausetets} feature these edges only once
        for any pair of opposite edges.

        Therefore edges
        $\mathit{QP} = \mathit{RY} = \mathit{ST}$,
        $\mathit{QR} = \mathit{PY} = \mathit{UT}$ and
        $\mathit{RS} = \mathit{YT} = \mathit{PU}$ need at least four
        markings between them from the external structure $\tau$.
        Since these edges are all joined to edges of types $b$ and $c$
        on the tori $\torus_1,\torus_2,\torus_3$, it follows that at least
        two of the boundary patterns on these tori must be $(0,2,0)$.
    \end{itemize}

    We have now established that, if there is any hope to extend $\tau$
    through the clause gadget, exactly one of the three boundary
    patterns on the three tori $\torus_1$, $\torus_2$ and $\torus_3$ must
    be $(2,0,0)$.  We must still show that such an extension is possible.

    Our clause gadget is symmetric, and so without loss of generality
    we can suppose that $\tau$ meets $\torus_1$ in
    the boundary pattern $(2,0,0)$, and meets $\torus_2$ and $\torus_3$
    in the boundary pattern $(0,2,0)$.  We mark edges within the clause
    gadget as follows:
    for the ``cone tetrahedra'' in Figure~\ref{fig:clausecone}, we mark
    edge pairs $\mathit{AB}/\mathit{CX}$ and
    $\mathit{DC}/\mathit{AX}$.
    For each of the two tetrahedra that we attach in
    Figure~\ref{fig:clausetets},
    we mark $\mathit{AC}$ and the corresponding opposite edge.

    From here it is a simple matter to follow the edge
    markings through the clause gadget and verify that every edge of the
    clause gadget receives exactly two markings in total;
    that is, $\tau$ can indeed be extended through the clause gadget as
    required.
\end{proof}

\subsection{Proving Theorem~\ref{thm:np}}
\label{sec:npproof}

Now that we are equipped with our various gadgets,
we can prove Theorem~\ref{thm:np}, i.e.,
that \problemname{taut angle structure} is \NP-complete.

In summary: we build a variable gadget for each variable $x_i$ of $\ms$,
duplicate its boundary torus using fork gadgets until we have one ``copy''
for each time $x_i$ occurs in a clause of $\ms$, and then hook these
boundary tori together using clause gadgets.
Lemmas~\ref{l:variable}, \ref{l:fork} and \ref{l:clause} together
ensure that the resulting triangulation (which is indeed orientable
with no boundary faces) has a taut angle structure if and only if
$\ms$ is solvable.
The full details are as follows.

\begin{proof}
    First we note that \problemname{taut angle structure}
    is clearly in \NP: if a taut angle structure exists, then the
    corresponding edge markings form a linear-sized certificate that
    is simple to verify in small polynomial time.

    To show that \problemname{taut angle structure} is
    \NP-complete, we give a polynomial reduction from
    \problemname{monotone 1-in-3 sat}.
    Let $\ms$ be an instance of \problemname{monotone 1-in-3 sat},
    as described at the beginning of Section~\ref{sec:np},
    with $t$ variables $x_1,\ldots,x_t$, and with $c$ clauses
    each of the form $x_i \vee x_j \vee x_k$.  For simplicity we assume
    that each variable appears in at least one clause (otherwise it can
    be harmlessly removed).
    We build a corresponding triangulation $\tri_\ms$ as follows:
    \begin{enumerate}[(i)]
        \item For each variable $x_i$, we construct a variable gadget $V_i$.

        \item For each variable $x_i$, suppose that $x_i$ appears
        $n_i$ times in total amongst the clauses of $\ms$ (so $\sum n_i = 3c$).
        Beginning with $V_i$, we attach a fork gadget to the boundary
        torus of $V_i$, then attach another fork gadget to one of the
        new boundary tori and so on, until we have attached $n_i-1$ fork
        gadgets in total.  Each time we attach a fork gadget we ensure
        that the boundary edge labels $a,b,c$ match, as described in
        Lemma~\ref{l:fork}.  The result is a connected triangulation
        with $n_i$ distinct two-triangle boundary tori; we denote this
        triangulation by $W_i$.

        \item For each clause $x_i \vee x_j \vee x_k$, we construct a
        clause gadget and attach one of the boundary tori from $W_i$,
        one of the boundary tori from $W_j$, and one of the boundary
        tori from $W_k$.  Again we ensure that the boundary edge labels
        $a,b,c$ match, as described in Lemma~\ref{l:clause}.
    \end{enumerate}

    By Observation~\ref{obs:lst} and Lemmas~\ref{l:fork} and
    \ref{l:clause}, the resulting object $\tri_\ms$ is a 3-manifold
    triangulation; moreover, it is simple to see from the construction
    that $\tri_\ms$ is orientable and has no remaining boundary faces.
    The total number of tetrahedra in $\tri_\ms$ is
    $2t + 21 \sum_{i=1}^t (n_i - 1) + 4c = 67c - 19t$, and the
    construction is easy to perform in small polynomial time in $t$ and $c$.

    All that remains is to show that $\tri_\ms$ has a taut angle
    structure if and only if $\ms$ is solvable:
    \begin{itemize}
        \item Suppose that $\ms$ is solvable.
        By Lemma~\ref{l:variable},
        we can assign a taut angle structure to each $V_i$ with
        boundary pattern $(2,0,0)$ or $(0,2,0)$ according to whether
        $x_i$ is true or false respectively.  By Lemma~\ref{l:fork},
        this extends to a taut angle structure on each $W_i$ where
        \emph{every} boundary pattern on $W_i$ is $(2,0,0)$ or
        $(0,2,0)$ according to whether $x_i$ is true or false respectively.
        Finally, because each clause contains exactly one true variable,
        Lemma~\ref{l:clause} shows that these taut angle structures
        extend through the clause gadgets, giving a taut angle
        structure on the full triangulation $\tri_\ms$.

        \item Suppose that $\tri_\ms$ has a taut angle structure $\tau$.
        By Lemma~\ref{l:variable}, restricting $\tau$ to each variable
        gadget $V_i$ must give one of the boundary patterns $(2,0,0)$ or
        $(0,2,0)$; we set the corresponding variable $x_i$ to true or false
        accordingly.  For each $i$, Lemma~\ref{l:fork} shows that
        $\tau$ must meet every boundary torus of $W_i$ in the
        same pattern as for $V_i$; that is, $(2,0,0)$ if we set $x_i$ to
        true, or $(0,2,0)$ if we set $x_i$ to false.
        Finally, because we know that $\tau$ extends through the clause
        gadgets, Lemma~\ref{l:clause} shows that each clause must have
        exactly one variable $x_i$ set to true; that is, $\ms$ is solvable.
    \end{itemize}

    Therefore our construction is indeed a polynomial reduction from
    \problemname{monotone 1-in-3 sat} to \problemname{taut angle structure},
    and so \problemname{taut angle structure} is \NP-complete.
\end{proof}

\section{Fixed-parameter tractability}
\label{sec:fpt}

So far we have shown that detecting taut angle structures is hard in general. 
However, in practice running times are often surprisingly fast.
This leads us to the natural question of
whether the running time
can be improved if we restrict ourselves to more 
specific classes of triangulations.

The way we approach this
question here is to prove that \problemname{taut angle structure}
is fixed-parameter tractable in both the cutwidth and the treewidth of
the face pairing graph of the triangulation.  The precise results are
given by Theorem~\ref{thm:cutwidth} (for cutwidth) and
Theorem~\ref{thm:treewidth} (for treewidth), both of which we restate below.
In this section we give full proofs for both of these theorems.

Note that the precise running times given here
(parameterised by both the number of tetrahedra $n$ and
the cutwidth/treewidth $k$) assume that we are given extra information
alongside our triangulation: for Theorem~\ref{thm:cutwidth} we assume
a left-to-right ordering (or layout) of nodes that corresponds to a
cutwidth of $\leq k$, and for Theorem~\ref{thm:treewidth} we assume a
tree decomposition of width $\leq k$ with $O(n)$ tree nodes.
In contrast, Corollary~\ref{cor:fpt} (that
\problemname{taut angle structure} is linear-time fixed-parameter
tractable in the treewidth) does not require any such information,
since in the setting of bounded treewidth
we can use the algorithm of Bodlaender \cite{bodlaender96-linear}
to compute such a tree decomposition in time linear in $n$.

In both proofs we assume that we have access to the full skeleton
of the triangulation $\tri$ (i.e., we know which tetrahedron
edges are identified and which tetrahedron vertices are identified);
such information is easily computed using linear time
depth-first search techniques.

\subsection{Bounded cutwidth: Proving Theorem \ref{thm:cutwidth}}

\setcounter{theorem}{1} 
\begin{theorem}
    Let $\tri$ be a 3-manifold triangulation with $n$ tetrahedra,
    where the graph $\fpg{\tri}$ has cutwidth $\leq k$, and for which a
    corresponding layout of nodes is known.
    Then \problemname{taut angle structure} can be solved for $\tri$ in
	$O(nk \cdot 3^{3k/2})$ time.
\end{theorem}

\begin{proof}
    Let the given layout of nodes of the face pairing graph
    $\fpg{\tri}$ be $v_1 , \ldots , v_n$,
    so that no cut $C_i = (\{v_1 , \ldots , v_i\},
    \ \{v_{i+1} , \ldots , v_{n}\})$ has width more than $k$.

    Recall that every node $v_i$ of $\Gamma (\tri)$ corresponds to a
    tetrahedron $\Delta_{v_i}$ of the triangulation $\tri$, and that every
    arc in the cutset for $C_i$ is an arc of $\fpg{\tri}$,
    and represents a triangle of $\tri$.
    
    Following a dynamic programming approach,
    we define sub-triangulations $\tri_1,\ldots,\tri_n$,
    where the sub-triangulation $\tri_i$ contains only the tetrahedra
    $\Delta_{v_1} , \ldots , \Delta_{v_i}$.
    We maintain all face gluings between these tetrahedra,
    but if a tetrahedron $\Delta_{v_x}$ is glued to a tetrahedron
    $\Delta_{v_y}$ in the full triangulation $\tri$ with $x \leq i < y$,
    then this will simply appear as a boundary face of
    $\Delta_{v_x}$ in $\tri_i$.

    A triangulation $\tri_i$ might contain ``pinched edges'',
    which occur when multiple edges on the boundary of the
    sub-triangulation $\tri_i$
    correspond to the same edge of $\tri$ (see edge $e$ in
    Figure~\ref{fig:pinch}).
    We happily accept such anomalies and consider these edges identical
    in $\tri_i$; as a result the number of edges on the boundary of
    $\tri_i$ is \emph{at most}
    (but might not be equal to) $3/2$ the number of faces.

    \begin{figure}[htb]
	\begin{center}
		\includegraphics[width=0.45\textwidth]{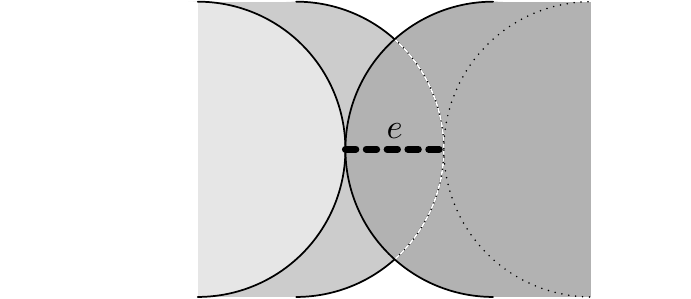}
	\end{center}
	\caption{A pinched edge in the boundary of $\tri_i$}
	\label{fig:pinch}
    \end{figure}

    By construction, every boundary face of $\tri_i$ corresponds to an
    arc in the cutset $C_i$, and so it follows that $\tri_i$ has
    $\leq k$ boundary faces and thus $\leq 3k/2$ boundary edges.
    Since any taut angle structure on $\tri_i$ must mark each boundary edge
    0, 1 or 2 times,
    there can be at most $3^{3k/2}$ different patterns of markings
    on the boundary of $\tri_i$ that correspond to taut angle structures
    on $\tri_i$.

    Following our dynamic programming strategy,
    we work through the triangulations in the order
    $\tri_1,\ldots,\tri_n$, and for each $i$ we compute precisely
    which boundary marking patterns on $\tri_i$ correspond to
    taut angle structures on $\tri_i$:
    \begin{itemize}
        \item For $\tri_1$ we simply try all three choices of markings
        on the tetrahedron $\Delta_{v_1}$, identify which of these
        form a taut angle structure on $\tri_1$, and store the resulting
        marking patterns on the boundary.

        \item For $\tri_i$, we consider each of the $\leq 3^{3k/2}$
        boundary marking patterns on $\tri_{i-1}$ that yields a
        taut angle structure on $\tri_{i-1}$, and attempt to combine
        these with each of the three choices of markings on the
        new tetrahedron $\Delta_{v_i}$.
        We discard any combination that marks an edge more than twice,
        or that marks an internal edge less than twice; the remaining
        combinations yield taut angle structures on $\tri_i$, and we
        store them in our solution set for $\tri_i$.
        See Figure~\ref{fig:cutWidth} for an illustration of this procedure.
    \end{itemize}

    \begin{figure}[htb]
        \begin{center}
            \includegraphics[width=.66\textwidth]{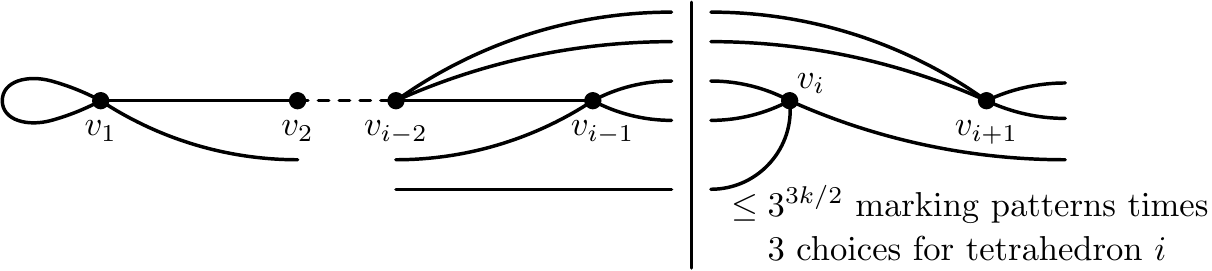}
        \end{center}
        \caption{Finding taut angle structures under bounded cutwidth}
        \label{fig:cutWidth}
    \end{figure}

    Since the final triangulation $\tri_n = \tri$ has no boundary
    faces (and hence no boundary edges), the full triangulation $\tri$
    has a taut angle structure if and only if the solution set for
    $\tri_n$ contains the empty marking pattern (as opposed to no
    marking patterns at all).

    Taking into account that in each step the number of boundary
    patterns to consider is at most $3^{3k/2}$, that there are only three
    choices of markings for the tetrahedron $\Delta_{v_i}$, and that we
    can update and test each boundary marking pattern in $O(k)$ time,
    it follows that each step can be performed in $O(k \cdot 3^{3k/2})$ time
    overall.  The total running time for all $n$ steps of the
    algorithm is therefore
    $O(nk \cdot 3^{3k/2})$.
\end{proof}

\subsection{Bounded treewidth: Proving Theorem \ref{thm:treewidth}}

\begin{theorem}
    Let $\tri$ be a 3-manifold triangulation with $n$ tetrahedra,
    where the graph $\fpg{\tri}$ has treewidth $\leq k$, and for which a
    corresponding tree decomposition with $O(n)$ tree nodes is known.
    Then \problemname{taut angle structure} can be solved for $\tri$ in
    $O(nk \cdot 3^{7k})$ time.
\end{theorem}

\begin{proof}
    Here we adopt a similar approach to before, but this time we do our dynamic
    programming over a tree.

    Recall that each node $\nu$ of the tree corresponds to a bag of
    nodes in $\fpg{\tri}$; that is, a bag of tetrahedra.
    We arbitrarily choose a root for the tree, so that the tree becomes a
    hierarchy of subtrees as illustrated in Figure~\ref{fig:treeWidth}.

    \begin{figure}[htb]
        \begin{center}
            \includegraphics[width=0.34\textwidth]{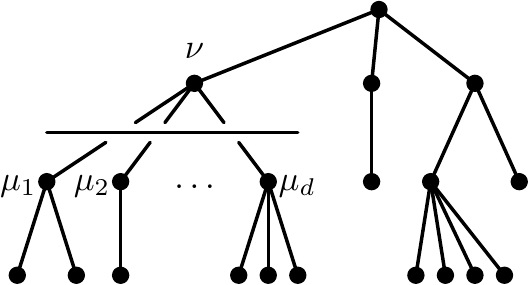}
        \end{center}
        \caption{Dynamic programming over a tree decomposition}
        \label{fig:treeWidth}
    \end{figure}

    As in the \problemname{taut angle structure} problem statement,
    we assume that $\tri$ has no boundary faces.
    For each node $\nu$ of the tree, we define the sub-triangulation
    $\tri_\nu$ by considering only those tetrahedra that appear only in
    bags within the subtree rooted at $\nu$; that is, we exclude
    any tetrahedron that appears in any bag outside this subtree.
    As in the previous proof,
    we maintain all face gluings between these tetrahedra;
    however, if a tetrahedron $\Delta \in \tri_\nu$ is glued to a tetrahedron
    $\Delta' \notin \tri_\nu$ in the full triangulation $\tri$ then this will
    simply appear as a boundary face of $\Delta$ in $\tri_\nu$.
    Once again a triangulation $\tri_\nu$ might contain ``pinched edges'';
    again we happily accept such anomalies and consider these edges identical
    in $\tri_\nu$, which means that
    the number of edges on the boundary of $\tri_\nu$ is at most
    (but not necessarily equal to) $3/2$ the number of faces.

    We now make a series of observations:
    \begin{enumerate}[(i)]
        \item \label{en:bdry}
        \emph{Each sub-triangulation $\tri_\nu$ has at most
        $4(k+1)$ boundary faces and at most $6(k+1)$ boundary edges.}

        If $\nu$ is the root node then this is trivial (since
        $\tri_\nu=\tri$ contains no boundary faces at all).
        Otherwise, let $\eta$ be the parent node of $\nu$ in the tree.
        Any boundary face of $\tri_\nu$ must correspond to some gluing between
        tetrahedra $\Delta \in \tri_\nu$
        and $\Delta' \notin \tri_\nu$, which in turn corresponds to some arc
        $\alpha$ in the face pairing graph.  Because
        $\Delta$ only appears in bags within the subtree rooted at $\nu$,
        and because one of these bags
        must contain both endpoints of the arc $\alpha$, it follows that
        $\Delta'$ appears in some bag within the subtree at $\nu$ also.

        By the definition of tree decomposition, since $\Delta'$ appears
        in some bag within the subtree at $\nu$ and also in some bag outside
        this subtree, it must also appear in the bag at the parent node $\eta$.
        There are $\leq k+1$ tetrahedra in the bag at $\eta$
        and so $\leq 4(k+1)$ possibilities for the face of $\Delta'$
        that is joined to $\Delta$.
        Therefore there are $\leq 4(k+1)$ such boundary faces of $\tri_\nu$.

        Finally, since the number of edges on the boundary surface is
        $\leq 3/2$ the number of faces,
        $\tri_\nu$ has $\leq 6(k+1)$ boundary edges.

        \item \label{en:disjoint}
        \emph{Suppose the tree node $\nu$ has child nodes
        $\mu_1,\ldots,\mu_d$, as illustrated in Figure~\ref{fig:treeWidth}.
        Then no two triangulations $\tri_{\mu_i},\tri_{\mu_j}$ have any
        tetrahedra in common.}

        This is true by definition of $\tri_{\mu_i}$, since no tetrahedron
        in $\tri_{\mu_i}$ can appear in the subtree rooted at
        $\mu_j$ (which lies outside the subtree rooted at $\mu_i$).

        \item \label{en:bdrysum}
        \emph{Suppose the tree node $\nu$ has child nodes
        $\mu_1,\ldots,\mu_d$, as illustrated in Figure~\ref{fig:treeWidth}.
        Then the triangulations $\tri_{\mu_1},\ldots,\tri_{\mu_d}$
        have $\leq 4(k+1)$ boundary faces in total.
        Moreover, they have $\leq 6(k+1)$ boundary edges in total, even
        if we count each edge repeatedly for each $\tri_{\mu_i}$
        that contains it.}

        As in the argument for (\ref{en:bdry}) above,
        each boundary face of each triangulation $\tri_{\mu_i}$
        corresponds to a face of some tetrahedron $\Delta'$ in the bag
        at the parent node $\nu$.  Moreover, from (\ref{en:disjoint})
        the triangulations $\tri_{\mu_i}$ contain distinct tetrahedra,
        and so each such boundary face can only appear in \emph{one} of the
        $\tri_{\mu_i}$.  Therefore the triangulations
        $\tri_{\mu_1},\ldots,\tri_{\mu_d}$ have
        $\leq 4(k+1)$ boundary faces between them, and these boundary
        faces are all distinct.
        
        From above, each $\tri_{\mu_j}$ has at most $3/2$ as many
        boundary edges as it has boundary faces.
        It therefore follows that $\tri_{\mu_1},\ldots,\tri_{\mu_d}$ have
        $\leq 6(k+1)$ boundary edges between them, even if we count
        repeated edges multiple times.
    \end{enumerate}

    Our algorithm for solving \problemname{taut angle structure}
    is based on dynamic programming over the tree, and operates as follows.
    For each triangulation $\tri_\nu$, we
    compute the number of marking patterns on the boundary
    of $\tri_\nu$ that correspond to taut angle structures on $\tri_\nu$.
    Since there are $\leq 6(k+1)$ boundary edges on $\tri_\nu$,
    there are $\leq 3^{6(k+1)}$ such possible marking patterns.

    We work our way from the leaves of the tree up to the root,
    computing these boundary patterns on each triangulation
    $\tri_\nu$ as we go:
    \begin{itemize}
        \item If $\nu$ is a leaf node, we simply try all
        $3^{k+1}$ possible markings on the $\leq k+1$ tetrahedra in the
        bag at $\nu$, which takes $O(3^{k+1})$ time.
        For each combination that yields a taut angle structure,
        we record the corresponding marking pattern on the boundary.

        \item If $\nu$ is not a leaf node then let $\mu_1,\ldots,\mu_d$ be
        its immediate children in the tree, as illustrated in
        Figure~\ref{fig:treeWidth}.  Let $\tri'$ be the
        (possibly disconnected) triangulation obtained by combining the
        tetrahedra from $\tri_{\mu_1},\ldots,\tri_{\mu_d}$.

        For each combination of boundary marking patterns on
        $\tri_{\mu_1},\ldots,\tri_{\mu_d}$, we combine these into a
        single boundary marking pattern on $\tri'$ (if any boundary edges
        are repeated then we sum the corresponding markings).
        By observation~(\ref{en:bdrysum})
        above, we can form each such combination in $O(k)$ time.
        If each $\tri_{\mu_i}$ has $b_i$ boundary edges,
        the total number of combinations that we form is
        $\leq 3^{b_1} \ldots 3^{b_d} = 3^{b_1+\ldots+b_d} \leq 3^{6(k+1)}$.

        We discard any combination that marks a boundary edge more than
        twice in total, or that marks an internal edge less than twice.
        Any combination $b'$ that survives must correspond to
        a taut angle structure on $\tri'$ (we simply combine the
        taut angle structures on each $\tri_{\mu_i}$, which we can do
        because the $\tri_{\mu_i}$ have no tetrahedra in common).
        We now combine $b'$ with all possible markings
        on the \emph{new} tetrahedra in $\tri_\nu$ that are not already
        present in $\tri'$; there are $\leq k+1$ new tetrahedra
        because they must all belong to the bag at $\nu$.
        Again we discard combinations that mark an edge more than twice,
        or that mark an internal edge less than twice;
        any combination that remains must arise from a taut angle structure
        on $\tri_\nu$, whereupon we add its boundary marking
        pattern to our solution set.

        In summary, we obtain $\leq 3^{6(k+1)}$ marking patterns
        on $\tri'$ which we combine with $3^{k+1}$ choices of markings
        for the new tetrahedra in $\tri_\nu$,
        giving a grand total of $\leq 3^{7(k+1)}$
        combinations overall.  Adding in a factor of $k$ to merge
        marking patterns and test for bad edge markings,
        the overall running time of this step is
        $O(k \cdot 3^{7(k+1)}) = O(k \cdot 3^{7k})$.
    \end{itemize}

    Let $\rho$ be the root node.  Since $\tri_\rho = \tri$ has no boundary
    faces (and hence no boundary edges), the full triangulation $\tri$
    has a taut angle structure if and only if the solution set for
    $\tri_\rho$ contains the empty marking pattern (as opposed to no
    marking patterns at all).
    Since there are $O(n)$ nodes in our tree decomposition,
    the total running time for the algorithm is
    $O(nk \cdot 3^{7k})$.
\end{proof}

\section{Discussion}
\label{sec:conc}

Theorem~\ref{thm:np} shows that even if we
restrict our attention to orientable triangulations with no boundary
faces, detecting taut angle structures is still \NP-complete.
However, our construction creates triangulations with many ideal
vertices.  It would be interesting to know if this \NP-completeness
result could be tightened to detecting taut angle structures on
\emph{one-vertex} triangulations.

We have an explicit script that uses the software package
\regina\ to build the triangulation $\tri_\ms$ for a given
\problemname{monotone 1-in-3 sat} instance $\ms$.
In the full version of this paper we discuss this script further
and describe some of the 3-manifolds that it produces.

Theorems~\ref{thm:cutwidth} and \ref{thm:treewidth} help explain why
taut angle structures are relatively easy to detect in practice
\cite{hodgson11-veering}:
there are many triangulations $\tri$ for which $\fpg{\tri}$
has small cutwidth and/or treewidth.\footnote{%
    In contrast, there are at present no conjectured examples of manifolds
    that do \emph{not} have small treewidth triangulations.  Such a
    conjecture would likely be extremely difficult to prove.}
In the closed setting, for
instance, the conjectured minimal triangulations of many Seifert fibred spaces
have extremely small treewidth \cite{martelli04-families,matveev98-or6},
and common building blocks
such as layered solid tori and triangular prisms have treewidth 1 and 3
respectively.
Small cutwidth and treewidth triangulations have also been found
fast to work with in
other settings, such as normal surface theory \cite{burton10-dd}.

It is worth considering whether we can find a more powerful parameter than
the treewidth of the face pairing graph.  In the more general setting of
constraint satisfaction problems, it is known (under certain hypotheses)
that treewidth essentially yields the best algorithms \cite{marx07-treewidth}.
In our setting, however, we are also subject to strong topological constraints
that are difficult to analyse in a purely combinatorial framework, and
that may provide new opportunities for optimisation.

Looking forward: the frameworks in this paper for \NP-completeness
and fixed-parameter tractability have significant potential
for use with \emph{normal surface theory}, which is the central
algorithmic machine for solving problems such as unknot recognition,
3-sphere recognition, prime decomposition, and many more.
A key feature of both taut angle structures and normal surfaces
is that they correspond to vertices of a high-dimensional polytope
subject to simple combinatorial constraints derived from the
tetrahedra \cite{kang05-taut2,luo08-angle-normal}.
Moreover, both taut angle structures and normal surfaces
can be incrementally ``extended'' through different sections of the
triangulation according to constraints derived from the local face
gluings---a technique used throughout this paper.

Although normal surfaces are more numerous and more difficult to work
with, this common foundation gives us hope that the techniques developed
here could, with further research, be used to tackle some of
the fundamental open complexity problems in knot theory and 3-manifold
topology.

\appendix

\section*{Appendix: The fork gadget}

In Section~\ref{sec:fork} we outline the construction of the fork
gadget, but we do not give the precise 21-tetrahedron triangulation
of the annular prism.  Here we present this 21-tetrahedron
triangulation in full.  We refer the reader to the labels and diagrams
from Section~\ref{sec:fork}.

Table~\ref{tab:forkprism} lists the individual face gluings for the
annular prism as shown in Figure~\ref{fig:forkprism}.
There are 21 tetrahedra labelled $\Delta_1,\ldots,\Delta_{21}$,
and the four vertices of each tetrahedron are labelled $1,2,3,4$.
Each row of the table represents a tetrahedron,
and each column represents one of its four faces.
For instance, the top-left cell of table indicates that
the face with vertices $1,2,3$ of tetrahedron $\Delta_1$ is glued to the
face with vertices $3,1,2$ (in that order) of tetrahedron $\Delta_{18}$
(the same gluing can be seen from the other side in the
fourth-last row of the table).

\begin{table}[ht]
\centering
\[\begin{array}{c|c|c|c|c}
&
\multicolumn{1}{r|}{\mathrm{Face}\ 123} &
\multicolumn{1}{r|}{\mathrm{Face}\ 124} &
\multicolumn{1}{r|}{\mathrm{Face}\ 134} &
\multicolumn{1}{r}{\mathrm{Face}\ 234} \\
\hline
\Delta_{1\phantom{0}} & \Delta_{18} : 312 & \Delta_{13} : 312 & \Delta_{18} : 324 & \Delta_{4\phantom{0}} : 234 \\
\Delta_{2\phantom{0}} & \Delta_{7\phantom{0}} : 342 & \Delta_{20} : 342 & \Delta_{20} : 312 & \Delta_{11} : 234 \\
\Delta_{3\phantom{0}} & \Delta_{19} : 312 & \Delta_{8\phantom{0}} : 342 & \Delta_{16} : 321 & \mbox{\em Outer L} \\
\Delta_{4\phantom{0}} & \mbox{\em Outer L} & \Delta_{9\phantom{0}} : 124 & \mbox{\em Outer R} & \Delta_{1\phantom{0}} : 234 \\
\Delta_{5\phantom{0}} & \Delta_{16} : 342 & \Delta_{15} : 321 & \Delta_{19} : 324 & \mbox{\em Outer R} \\
\Delta_{6\phantom{0}} & \mbox{\em Inner} & \Delta_{12} : 342 & \Delta_{21} : 312 & \mbox{\em Upper} \\
\Delta_{7\phantom{0}} & \mbox{\em Inner} & \Delta_{11} : 132 & \Delta_{8\phantom{0}} : 134 & \Delta_{2\phantom{0}} : 312 \\
\Delta_{8\phantom{0}} & \Delta_{14} : 324 & \Delta_{14} : 321 & \Delta_{7\phantom{0}} : 134 & \Delta_{3\phantom{0}} : 412 \\
\Delta_{9\phantom{0}} & \mbox{\em Upper} & \Delta_{4\phantom{0}} : 124 & \mbox{\em Upper} & \Delta_{10} : 234 \\
\Delta_{10} & \Delta_{13} : 324 & \Delta_{11} : 124 & \Delta_{11} : 134 & \Delta_{9\phantom{0}} : 234 \\
\Delta_{11} & \Delta_{7\phantom{0}} : 142 & \Delta_{10} : 124 & \Delta_{10} : 134 & \Delta_{2\phantom{0}} : 234 \\
\Delta_{12} & \Delta_{21} : 423 & \Delta_{13} : 124 & \Delta_{13} : 134 & \Delta_{6\phantom{0}} : 412 \\
\Delta_{13} & \Delta_{1\phantom{0}} : 241 & \Delta_{12} : 124 & \Delta_{12} : 134 & \Delta_{10} : 213 \\
\Delta_{14} & \Delta_{8\phantom{0}} : 421 & \Delta_{15} : 124 & \Delta_{15} : 134 & \Delta_{8\phantom{0}} : 213 \\
\Delta_{15} & \Delta_{5\phantom{0}} : 421 & \Delta_{14} : 124 & \Delta_{14} : 134 & \mbox{\em Lower} \\
\Delta_{16} & \Delta_{3\phantom{0}} : 431 & \Delta_{17} : 124 & \Delta_{17} : 134 & \Delta_{5\phantom{0}} : 312 \\
\Delta_{17} & \mbox{\em Lower} & \Delta_{16} : 124 & \Delta_{16} : 134 & \mbox{\em Lower} \\
\Delta_{18} & \Delta_{1\phantom{0}} : 231 & \Delta_{19} : 124 & \Delta_{19} : 134 & \Delta_{1\phantom{0}} : 314 \\
\Delta_{19} & \Delta_{3\phantom{0}} : 231 & \Delta_{18} : 124 & \Delta_{18} : 134 & \Delta_{5\phantom{0}} : 314 \\
\Delta_{20} & \Delta_{2\phantom{0}} : 341 & \Delta_{21} : 124 & \Delta_{21} : 134 & \Delta_{2\phantom{0}} : 412 \\
\Delta_{21} & \Delta_{6\phantom{0}} : 341 & \Delta_{20} : 124 & \Delta_{20} : 134 & \Delta_{12} : 231 \\
\end{array}\]
\caption{The 21-tetrahedron triangulation of the prism over the annulus}
\label{tab:forkprism}
\end{table}

There are 12 faces on the boundary of the annular prism:
three on the upper annulus with vertices $A,B,E$ (marked
\emph{Upper} in the table);
three on the lower annulus with vertices $C,D,F$ (marked
\emph{Lower} in the table);
four on the outer cylinder with vertices $A,B,C,D$ (two
on the left-hand side of the diagram marked \emph{Outer L}, and two
on the right-hand side of the diagram marked \emph{Outer R}),
and finally two on the inner cylinder with vertices $E,F$
(marked \emph{Inner} in the table).

The final stage of the construction is to glue the upper annulus to the
lower annulus.  The result is shown in Table~\ref{tab:forkfull},
where only six boundary faces remain (on the outer and inner cylinders).
This triangulation is the fork gadget in its entirety.

\begin{table}[ht]
\centering
\[\begin{array}{c|c|c|c|c}
&
\multicolumn{1}{r|}{\mathrm{Face}\ 123} &
\multicolumn{1}{r|}{\mathrm{Face}\ 124} &
\multicolumn{1}{r|}{\mathrm{Face}\ 134} &
\multicolumn{1}{r}{\mathrm{Face}\ 234} \\
\hline
\Delta_{1\phantom{0}} & \Delta_{18} : 312 & \Delta_{13} : 312 & \Delta_{18} : 324 & \Delta_{4\phantom{0}} : 234 \\
\Delta_{2\phantom{0}} & \Delta_{7\phantom{0}} : 342 & \Delta_{20} : 342 & \Delta_{20} : 312 & \Delta_{11} : 234 \\
\Delta_{3\phantom{0}} & \Delta_{19} : 312 & \Delta_{8\phantom{0}} : 342 & \Delta_{16} : 321 & \mbox{\em Outer L} \\
\Delta_{4\phantom{0}} & \mbox{\em Outer L} & \Delta_{9\phantom{0}} : 124 & \mbox{\em Outer R} & \Delta_{1\phantom{0}} : 234 \\
\Delta_{5\phantom{0}} & \Delta_{16} : 342 & \Delta_{15} : 321 & \Delta_{19} : 324 & \mbox{\em Outer R} \\
\Delta_{6\phantom{0}} & \mbox{\em Inner} & \Delta_{12} : 342 & \Delta_{21} : 312 & \Delta_{15} : 432 \\
\Delta_{7\phantom{0}} & \mbox{\em Inner} & \Delta_{11} : 132 & \Delta_{8\phantom{0}} : 134 & \Delta_{2\phantom{0}} : 312 \\
\Delta_{8\phantom{0}} & \Delta_{14} : 324 & \Delta_{14} : 321 & \Delta_{7\phantom{0}} : 134 & \Delta_{3\phantom{0}} : 412 \\
\Delta_{9\phantom{0}} & \Delta_{17} : 213 & \Delta_{4\phantom{0}} : 124 & \Delta_{17} : 234 & \Delta_{10} : 234 \\
\Delta_{10} & \Delta_{13} : 324 & \Delta_{11} : 124 & \Delta_{11} : 134 & \Delta_{9\phantom{0}} : 234 \\
\Delta_{11} & \Delta_{7\phantom{0}} : 142 & \Delta_{10} : 124 & \Delta_{10} : 134 & \Delta_{2\phantom{0}} : 234 \\
\Delta_{12} & \Delta_{21} : 423 & \Delta_{13} : 124 & \Delta_{13} : 134 & \Delta_{6\phantom{0}} : 412 \\
\Delta_{13} & \Delta_{1\phantom{0}} : 241 & \Delta_{12} : 124 & \Delta_{12} : 134 & \Delta_{10} : 213 \\
\Delta_{14} & \Delta_{8\phantom{0}} : 421 & \Delta_{15} : 124 & \Delta_{15} : 134 & \Delta_{8\phantom{0}} : 213 \\
\Delta_{15} & \Delta_{5\phantom{0}} : 421 & \Delta_{14} : 124 & \Delta_{14} : 134 & \Delta_{6\phantom{0}} : 432 \\
\Delta_{16} & \Delta_{3\phantom{0}} : 431 & \Delta_{17} : 124 & \Delta_{17} : 134 & \Delta_{5\phantom{0}} : 312 \\
\Delta_{17} & \Delta_{9\phantom{0}} : 213 & \Delta_{16} : 124 & \Delta_{16} : 134 & \Delta_{9\phantom{0}} : 134 \\
\Delta_{18} & \Delta_{1\phantom{0}} : 231 & \Delta_{19} : 124 & \Delta_{19} : 134 & \Delta_{1\phantom{0}} : 314 \\
\Delta_{19} & \Delta_{3\phantom{0}} : 231 & \Delta_{18} : 124 & \Delta_{18} : 134 & \Delta_{5\phantom{0}} : 314 \\
\Delta_{20} & \Delta_{2\phantom{0}} : 341 & \Delta_{21} : 124 & \Delta_{21} : 134 & \Delta_{2\phantom{0}} : 412 \\
\Delta_{21} & \Delta_{6\phantom{0}} : 341 & \Delta_{20} : 124 & \Delta_{20} : 134 & \Delta_{12} : 231 \\
\end{array}\]
\caption{The complete triangulation of the fork gadget}
\label{tab:forkfull}
\end{table}

For reference, the six boundary faces as shown in
Figure~\ref{fig:forkglue} are as follows:
\begin{itemize}
    \item the upper triangle $\mathit{ABD}$ on the left-hand side of the outer
    cylinder is $\Delta_{4} : 213$;
    \item the lower triangle $\mathit{ACD}$ on the left-hand side of the outer
    cylinder is $\Delta_{3} : 243$;
    \item the upper triangle $\mathit{ABD}$ on the right-hand side of the outer
    cylinder is $\Delta_{4} : 413$;
    \item the lower triangle $\mathit{ACD}$ on the right-hand side of the outer
    cylinder is $\Delta_{5} : 423$;
    \item the upper triangle $\mathit{EFE}$ on the inner cylinder
    is $\Delta_{6} : 312$;
    \item the lower triangle $\mathit{FEF}$ on the inner cylinder
    is $\Delta_{7} : 321$.
\end{itemize}

%
%

\section*{Acknowledgements}

The first author is grateful to the Australian Research Council for their
support under the Discovery Projects funding scheme
(projects DP1094516 and DP110101104).

%
%

\small
\bibliographystyle{amsplain}
\bibliography{pure}

\newcommand{\noopsort}[1]{}
\providecommand{\bysame}{\leavevmode\hbox to3em{\hrulefill}\thinspace}
\providecommand{\MR}{\relax\ifhmode\unskip\space\fi MR }
\providecommand{\MRhref}[2]{%
  \href{http://www.ams.org/mathscinet-getitem?mr=#1}{#2}
}
\providecommand{\href}[2]{#2}
\begin{thebibliography}{10}

\bibitem{agol06-knotgenus}
Ian Agol, Joel Hass, and William Thurston, \emph{The computational complexity
  of knot genus and spanning area}, Trans. Amer. Math. Soc. \textbf{358}
  (2006), no.~9, 3821--3850 (electronic).

\bibitem{arnborg87-embeddings}
Stefan Arnborg, Derek~G. Corneil, and Andrzej Proskurowski, \emph{Complexity of
  finding embeddings in a {$k$}-tree}, SIAM J. Algebraic Discrete Methods
  \textbf{8} (1987), no.~2, 277--284.

\bibitem{bodlaender86-classes}
Hans~L. Bodlaender, \emph{Classes of graphs with bounded tree-width}, Technical
  Report RUU-CS-86-22, Utrecht University, 1986.

\bibitem{bodlaender96-linear}
\bysame, \emph{A linear-time algorithm for finding tree-decompositions of small
  treewidth}, SIAM J. Comput. \textbf{25} (1996), no.~6, 1305--1317.

\bibitem{burton04-facegraphs}
Benjamin~A. Burton, \emph{Face pairing graphs and 3-manifold enumeration}, J.
  Knot Theory Ramifications \textbf{13} (2004), no.~8, 1057--1101.

\bibitem{burton04-regina}
\bysame, \emph{Introducing {R}egina, the 3-manifold topology software},
  Experiment. Math. \textbf{13} (2004), no.~3, 267--272.

\bibitem{burton10-dd}
\bysame, \emph{Optimizing the double description method for normal surface
  enumeration}, Math. Comp. \textbf{79} (2010), no.~269, 453--484.

\bibitem{regina}
Benjamin~A. Burton, Ryan Budney, William Pettersson, et~al., \emph{Regina:
  Software for 3-manifold topology and normal surface theory},
  \texttt{http://\allowbreak regina.\allowbreak sourceforge.\allowbreak net/},
  1999--2012.

\bibitem{downey99-param}
R.~G. Downey and M.~R. Fellows, \emph{Parameterized complexity}, Monographs in
  Computer Science, Springer-Verlag, New York, 1999.

\bibitem{dunfield11-spanning}
Nathan~M. Dunfield and Anil~N. Hirani, \emph{The least spanning area of a knot
  and the optimal bounding chain problem}, {SCG} '11: Proceedings of the
  Twenty-Seventh Annual Symposium on Computational Geometry, ACM, 2011,
  pp.~135--144.

\bibitem{futer11-angled}
David Futer and Fran{\c{c}}ois Gu{\'e}ritaud, \emph{From angled triangulations
  to hyperbolic structures}, Interactions Between Hyperbolic Geometry, Quantum
  Topology and Number Theory, Contemp. Math., vol. 541, Amer. Math. Soc.,
  Providence, RI, 2011, pp.~159--182.

\bibitem{garey79-intractability}
Michael~R. Garey and David~S. Johnson, \emph{Computers and intractability}, W.
  H. Freeman and Co., San Francisco, Calif., 1979, A guide to the theory of
  NP-completeness, A Series of Books in the Mathematical Sciences.

\bibitem{hass12-conp}
Joel Hass, \emph{New results on the complexity of recognizing the 3-sphere}, To
  appear in Oberwolfach Rep., 2012.

\bibitem{hass99-knotnp}
Joel Hass, Jeffrey~C. Lagarias, and Nicholas Pippenger, \emph{The computational
  complexity of knot and link problems}, J. Assoc. Comput. Mach. \textbf{46}
  (1999), no.~2, 185--211.

\bibitem{hodgson11-veering}
Craig~D. Hodgson, J.~Hyam Rubinstein, Henry Segerman, and Stephan Tillmann,
  \emph{Veering triangulations admit strict angle structures}, Geom. Topol.
  \textbf{15} (2011), no.~4, 2073--2089.

\bibitem{jaco05-lectures-homeomorphism}
William Jaco, \emph{The homeomorphism problem: Classification of 3-manifolds},
  Lecture notes, Available from \texttt{http://\allowbreak www.\allowbreak
  math.\allowbreak okstate.\allowbreak edu/\allowbreak \~{}jaco/\allowbreak
  pekinglectures.\allowbreak htm}, 2005.

\bibitem{jaco03-0-efficiency}
William Jaco and J.~Hyam Rubinstein, \emph{0-efficient triangulations of
  3-manifolds}, J. Differential Geom. \textbf{65} (2003), no.~1, 61--168.

\bibitem{kang05-taut2}
Ensil Kang and J.~Hyam Rubinstein, \emph{Ideal triangulations of 3-manifolds
  {II}; {T}aut and angle structures}, Algebr. Geom. Topol. \textbf{5} (2005),
  1505--1533.

\bibitem{kleiner08-perelman}
Bruce Kleiner and John Lott, \emph{Notes on {P}erelman's papers}, Geom. Topol.
  \textbf{12} (2008), no.~5, 2587--2855.

\bibitem{korach93-width}
Ephraim Korach and Nir Solel, \emph{Tree-width, path-width, and cutwidth},
  Discrete Appl. Math. \textbf{43} (1993), no.~1, 97--101.

\bibitem{kuperberg11-conp}
Greg Kuperberg, \emph{Knottedness is in {NP}, modulo {GRH}}, Preprint,
  \texttt{arXiv:\allowbreak 1112.0845}, November 2011.

\bibitem{lackenby00-taut}
Marc Lackenby, \emph{Taut ideal triangulations of 3-manifolds}, Geom. Topol.
  \textbf{4} (\noopsort{2000b}2000), 369--395.

\bibitem{luo08-angle-normal}
Feng Luo and Stephan Tillmann, \emph{Angle structures and normal surfaces},
  Trans. Amer. Math. Soc. \textbf{360} (2008), no.~6, 2849--2866.

\bibitem{martelli04-families}
Bruno Martelli and Carlo Petronio, \emph{Complexity of geometric
  three-manifolds}, Geom. Dedicata \textbf{108} (2004), no.~1, 15--69.

\bibitem{marx07-treewidth}
Daniel Marx, \emph{Can you beat treewidth?}, Proceedings of the 48th Annual
  {IEEE} Symposium on Foundations of Computer Science, FOCS '07, IEEE Computer
  Society, 2007, pp.~169--179.

\bibitem{matveev98-or6}
Sergei~V. Matveev, \emph{Tables of 3-manifolds up to complexity 6},
  Max-Planck-Institut f{\"u}r Mathematik Preprint Series (1998), no.~67,
  available from \texttt{http://www.\allowbreak mpim-bonn.\allowbreak
  mpg.\allowbreak de/\allowbreak html/\allowbreak pre\-prints/\allowbreak
  preprints.html}.

\bibitem{rivin94-structures}
Igor Rivin, \emph{Euclidean structures on simplicial surfaces and hyperbolic
  volume}, Ann. of Math. (2) \textbf{139} (1994), no.~3, 553--580.

\bibitem{rivin03-combopt}
\bysame, \emph{Combinatorial optimization in geometry}, Adv. in Appl. Math.
  \textbf{31} (2003), no.~1, 242--271.

\bibitem{rubinstein95-3sphere}
J.~Hyam Rubinstein, \emph{An algorithm to recognize the {$3$}-sphere},
  Proceedings of the International Congress of Mathematicians ({Z}{\"u}rich,
  1994), vol.~1, Birkh{\"a}user, 1995, pp.~601--611.

\bibitem{schaefer78-sat}
Thomas~J. Schaefer, \emph{The complexity of satisfiability problems},
  Conference {R}ecord of the {T}enth {A}nnual {ACM} {S}ymposium on {T}heory of
  {C}omputing ({S}an {D}iego, {C}alif., 1978), ACM, 1978, pp.~216--226.

\bibitem{schleimer11-np}
Saul Schleimer, \emph{Sphere recognition lies in {NP}}, Low-dimensional and
  Symplectic Topology (Michael Usher, ed.), Proceedings of Symposia in Pure
  Mathematics, vol.~82, Amer. Math. Soc., 2011, pp.~183--214.

\bibitem{thilikos05-cutwidth}
Dimitrios~M. Thilikos, Maria Serna, and Hans~L. Bodlaender, \emph{Cutwidth {I}:
  A linear time fixed parameter algorithm}, J. Algorithms \textbf{56} (2005),
  no.~1, 1--24.

\bibitem{thurston78-lectures}
William~P. Thurston, \emph{The geometry and topology of 3-manifolds}, Lecture
  notes, Princeton University, 1978.

\end{thebibliography}

%
%

\bigskip
\noindent
Benjamin A.~Burton \\
School of Mathematics and Physics, The University of Queensland \\
Brisbane QLD 4072, Australia \\
(bab@maths.uq.edu.au)

\bigskip
\noindent
Jonathan Spreer \\
School of Mathematics and Physics, The University of Queensland \\
Brisbane QLD 4072, Australia \\
(j.spreer@uq.edu.au)

\end{document}